\documentclass[english,11pt]{article}

\usepackage[utf8]{inputenc}
\usepackage[T1]{fontenc}
\usepackage[english]{babel}
\usepackage[margin=2cm]{geometry}
\usepackage{setspace}
\usepackage{amsmath,amssymb,amsthm}
\usepackage[shortlabels]{enumitem}
\usepackage{lmodern}
\usepackage{microtype}
\usepackage{hyperref}

\theoremstyle{plain}
\newtheorem{theorem}{Theorem}[section]
\newtheorem{lemma}[theorem]{Lemma}
\newtheorem{proposition}[theorem]{Proposition}
\newtheorem{corollary}[theorem]{Corollary}

\theoremstyle{definition}
\newtheorem{definition}[theorem]{Definition}
\theoremstyle{remark}
\newtheorem{remark}[theorem]{Remark}

\newcommand{\Pcal}{\mathcal P}
\newcommand{\LL}{\mathcal L}
\newcommand{\id}{\mathrm{id}}

\DeclareMathOperator{\Lip}{Lip}

\newcommand{\Aatt}{\mathcal A}

\title{Instability of the ray-monotone selector for \(W_1\)-optimal transport}
\author{Maja Gw\'{o}\'{z}d\'{z}\\[0.3em]\normalsize ETH Z\"urich\\\normalsize\texttt{mgwozdz@ethz.ch}}
\date{}

\begin{document}

\maketitle

\begin{abstract}
For the distance cost \(c(x,y)=|x-y|\), the set \(O(\mu,\nu)\) of \(W_1\)-optimal plans is generally not a singleton. Under the classical absolute-continuity hypotheses in the Euclidean case, secondary variational selection by the quadratic energy \(C_2\) yields the ray-monotone \(W_1\)-optimal plan. We provide a counterexample to an open problem posed by Santambrogio that concerns the stability of this selector under weak convergence of the marginals. More precisely, we construct a fixed absolutely continuous source \(\mu\) and absolutely continuous targets \(\nu_n\rightharpoonup\nu\) such that \(\gamma^{\mathrm{sel}}(\mu,\nu_n)\rightharpoonup\gamma^{\mathrm{hom}}\), where \(\gamma^{\mathrm{hom}}\in O(\mu,\nu)\) but \(\gamma^{\mathrm{hom}}\neq\gamma^{\mathrm{sel}}(\mu,\nu)\). We also identify the narrow Kuratowski limit of the optimal-plan sets \(O(\mu,\nu_n)\), derive the constrained \(\Gamma\)-limit for secondary energies of the form \(\int \Phi(|x-y|)\,d\gamma\) with \(\Phi\in C([0,2])\), and deduce a non-commutation result for the additive perturbation \(c_\varepsilon(x,y)=|x-y|+\varepsilon|x-y|^2\).
\end{abstract}



\section{Introduction}

For the distance cost \(c(x,y)=|x-y|\), the set \(O(\mu,\nu)\) of \(W_1\)-optimal plans is generally not a singleton. A common way to select a distinguished optimiser is to minimise, over \(O(\mu,\nu)\), a secondary functional with a strictly convex integrand. In the Euclidean setting, under the usual absolute-continuity hypotheses, this method yields the ray-monotone \(W_1\)-optimal plan. Secondary variational selection in the \(W_1\) Monge problem originates in \cite{AmbrosioPratelli} and was developed further in \cite{Caravenna2011,ChampionDePascale2010,ChampionDePascale2011}. In the Euclidean absolutely continuous setting, the fact that different strictly convex \(C^1\) secondary integrands lead to the same selector is stated in \cite[Theorem~5.5]{ChenLiuYang2021}. See also \cite[Section~3.1.4]{Santambrogio2015} for the ray-monotone interpretation.

A natural question is whether this selector is stable under narrow convergence of the marginals. This question was posed as an open problem in \cite[Section~3.1.5]{Santambrogio2015}: If \(\mu_n\rightharpoonup\mu\), \(\nu_n\rightharpoonup\nu\), and \(\gamma_n\) is the ray-monotone \(W_1\)-optimal plan between \(\mu_n\) and \(\nu_n\), must every narrow limit of \((\gamma_n)\) be the ray-monotone \(W_1\)-optimal plan between \(\mu\) and \(\nu\)? Our main theorem gives a negative answer, which already holds for a fixed absolutely continuous source and absolutely continuous targets that converge narrowly to an absolutely continuous limit.

We emphasise that our problem is distinct from the fixed-pair perturbative selection results obtained by modifying the cost, such as the \(r^{1+\varepsilon}\)-approximation in \cite{AmbrosioPratelli}, and also different from convergence results along special approximation schemes, such as the discrete-target approximation discussed in \cite[p.~350]{Santambrogio2009} and \cite[Remark~1.2]{Caravenna2011}. It is also orthogonal to the stability theory for quadratic-cost optimal transport maps, where both positive quantitative results and instability mechanisms are now available (see \cite{DelalandeMerigot2023,Letrouit2025}). In a similar vein, one- and higher-dimensional results on entropic regularisation of the distance cost address the vanishing-regularisation limit for a fixed pair of marginals, not the weak continuity of the secondary \(W_1\)-selector under varying marginals (see \cite{DiMarinoLouet2018,AryanGhosal2025,Ley2025}).

We fix
\[
K:=[-1,1]\times[0,1],
\qquad
\mu:=\LL^2\!\llcorner[0,1]^2,
\]
and vary only the target. For \(\nu\in\Pcal(K)\), let \(\Pi(\mu,\nu)\) be the set of couplings between \(\mu\) and \(\nu\), and set
\[
C_p(\gamma):=\int_{K\times K}|x-y|^p\,d\gamma,\qquad
m_1(\mu,\nu):=\min_{\gamma\in\Pi(\mu,\nu)} C_1(\gamma).
\]
In this compact setting, this is equivalent to
\[
O(\mu,\nu):=\{\gamma\in\Pi(\mu,\nu):C_1(\gamma)=m_1(\mu,\nu)\}.
\]
Whenever the secondary problem
\[
\min\{C_2(\gamma):\gamma\in O(\mu,\nu)\}
\]
has a unique minimiser, we denote it by \(\gamma^{\mathrm{sel}}(\mu,\nu)\). In the classical Euclidean absolutely continuous case, this unique secondary minimiser and the unique ray-monotone \(W_1\)-optimal plan coincide. We refer to \(\gamma^{\mathrm{sel}}(\mu,\nu)\) as the ray-monotone selector whenever the secondary minimiser is unique (see \cite[Section~3.1.4 and Section~3.1.5]{Santambrogio2015} and \cite[p.~350]{Santambrogio2009}).

The mechanism we apply in the proof relies on the a failure of commutation between fibrewise secondary selection and weak averaging in the transverse variable. On each oscillatory strip, the fibrewise quadratic selector is either the identity or the unit horizontal shift \(x_1\mapsto x_1-1\), whereas the averaged fibre target selects the monotone rearrangement \(x_1\mapsto 2x_1-1\). This produces a homogenised optimal plan that is not the selected plan for the limit pair. In Section~\ref{sec:variational}, we identify the constrained \(\Gamma\)-limit on the attainable class \(\Aatt\).

\begin{definition}[Oscillatory laminate family]\label{def:intro-osc}
With \(K=[-1,1]\times[0,1]\) and \(\mu=\LL^2\!\llcorner[0,1]^2\) as above, for \(n\in\mathbb N\) we set
\[
A_n:=\bigcup_{k=0}^{n-1}\Bigl[\frac{k}{n},\frac{2k+1}{2n}\Bigr),
\qquad
B_n:=[0,1]\setminus A_n,
\]
so that \(\LL^1(A_n)=\LL^1(B_n)=\frac12\). We define Borel maps \(T_n,S:[0,1]^2\to K\) by
\[
T_n(x_1,x_2):=
\begin{cases}
(x_1,x_2), & x_2\in A_n,\\
(x_1-1,x_2), & x_2\in B_n,
\end{cases}
\qquad
S(x_1,x_2):=(x_1-1,x_2).
\]
Set
\[
\nu_n:=(T_n)_\#\mu,
\qquad
\gamma_n:=(\id\times T_n)_\#\mu.
\]
We further define the limiting target and the homogenised plan as follows
\[
\nu:=\frac12(\id)_\#\mu+\frac12 S_\#\mu,
\qquad
\gamma^{\mathrm{hom}}:=\frac12(\id,\id)_\#\mu+\frac12(\id,S)_\#\mu.
\]
\end{definition}

Let us write
\[
\Delta:=\{(x,y)\in K\times K:\ x=y\},
\qquad
\Aatt:=\{\eta\in O(\mu,\nu):\eta(\Delta)=\tfrac12\}.
\]
We call \(\Aatt\) the \textit{attainable class}.

\begin{theorem}[Instability of the ray-monotone selector]\label{thm:main}
Let \((\nu_n)_{n\in\mathbb N}\), \((\gamma_n)_{n\in\mathbb N}\), \(\nu\), and \(\gamma^{\mathrm{hom}}\) be as in Definition~\ref{def:intro-osc}. It follows that for every \(n\),
\[
\gamma_n=\gamma^{\mathrm{sel}}(\mu,\nu_n).
\]
Moreover,
\[
\gamma_n\rightharpoonup\gamma^{\mathrm{hom}}\in O(\mu,\nu),
\qquad
\gamma^{\mathrm{hom}}(\Delta)=\tfrac12.
\]
For the limit pair,
\[
\gamma^{\mathrm{sel}}(\mu,\nu)=(\id\times T)_\#\mu,
\qquad
T(x_1,x_2)=(2x_1-1,x_2).
\]
In particular, it holds that
\[
\gamma^{\mathrm{sel}}(\mu,\nu)(\Delta)=0,
\qquad
C_2(\gamma^{\mathrm{hom}})=\frac12>\frac13=C_2\bigl(\gamma^{\mathrm{sel}}(\mu,\nu)\bigr).
\]
Finally, we obtain
\[
\gamma_n\rightharpoonup\gamma^{\mathrm{hom}}\neq \gamma^{\mathrm{sel}}(\mu,\nu).
\]
\end{theorem}

\begin{remark}
We fix the source measure \(\mu\) in the construction, and all targets \(\nu_n\) and \(\nu\) are absolutely continuous. This implies that the instability showed in Theorem~\ref{thm:main} is caused by the selector itself, and not the singular marginals.
\end{remark}

\begin{theorem}[Kuratowski limit of the optimal-plan sets]\label{thm:intro-stable}
Under the hypotheses of Theorem~\ref{thm:main}, the narrow Kuratowski upper and lower limits of \(O(\mu,\nu_n)\) (see Definition~\ref{def:stable-closure}) coincide and satisfy
\[
\liminf_{n\to\infty} O(\mu,\nu_n)
=
\limsup_{n\to\infty} O(\mu,\nu_n)
=
\Aatt
=
\{\eta\in O(\mu,\nu):\eta(\Delta)=\tfrac12\}.
\]
\end{theorem}

\begin{theorem}[Constrained \(\Gamma\)-convergence]\label{thm:main-gamma}
Let \(\Phi:[0,2]\to\mathbb R\) be continuous and define
\[
\mathcal F_n(\eta):=
\begin{cases}
\displaystyle \int_{K\times K}\Phi(|x-y|)\,d\eta, & \eta\in O(\mu,\nu_n),\\[0.3em]
+\infty, & \text{otherwise}.
\end{cases}
\]
It follows that the sequence \((\mathcal F_n)\) \(\Gamma\)-converges on \((\Pcal(K\times K),\rightharpoonup)\) to
\[
\mathcal F_\infty(\eta):=
\begin{cases}
\displaystyle \int_{K\times K}\Phi(|x-y|)\,d\eta, & \eta\in\Aatt,\\[0.3em]
+\infty, & \text{otherwise}.
\end{cases}
\]
\end{theorem}

Theorems~\ref{thm:intro-stable} and~\ref{thm:main-gamma} show that the instability is not simply a failure of pointwise selector convergence. Indeed, the whole family of optimal-plan sets has an effective limit, and the limiting secondary problem lives on the proper subset \(\Aatt\subset O(\mu,\nu)\).

\begin{corollary}[Failure of commutation for the additive quadratic perturbation]\label{cor:quadratic}
For each \(n\in\mathbb N\) and \(\varepsilon>0\), let \(\gamma_{n,\varepsilon}\in\Pi(\mu,\nu_n)\) be any \(c_\varepsilon\)-optimal plan, where
\[
c_\varepsilon(x,y):=|x-y|+\varepsilon |x-y|^2.
\]
Narrowly in \(\Pcal(K\times K)\), it holds that
\[
\lim_{n\to\infty}\lim_{\varepsilon\downarrow0}\gamma_{n,\varepsilon}
=
\gamma^{\mathrm{hom}}
\neq
\gamma^{\mathrm{sel}}(\mu,\nu)
=
\lim_{\varepsilon\downarrow0}\lim_{n\to\infty}\gamma_{n,\varepsilon}.
\]
\end{corollary}

\begin{remark}
The same construction extends, with only notational changes, to every dimension \(d\ge2\) by tensoring with the identity in the remaining coordinates.
\end{remark}

The proof has four steps. We begin by identifying a common contact set for all \(W_1\)-optimal plans in the laminate family. We then reduce admissible plans fibrewise, and solve the one-dimensional secondary problems. Finally, we pass to the oscillatory limit and identify the effective attainable class. In Section~\ref{sec:prelim}, we develop the fibrewise description of admissible plans for the laminate family. In Section~\ref{sec:counterexample}, we prove Theorems~\ref{thm:main} and~\ref{thm:intro-stable}. Section~\ref{sec:variational} shows Theorem~\ref{thm:main-gamma} and Corollary~\ref{cor:quadratic}. In Appendix~\ref{sec:appendix}, we present the auxiliary recovery and averaging lemmas that we frequently use in the main proof.

\section{Fibrewise framework for the laminate family}\label{sec:prelim}

We interpret all measures are Borel probability measures unless explicitly stated otherwise. We write \(\alpha_n\rightharpoonup\alpha\) for narrow convergence. Since \(K\) is compact, narrow convergence is equivalent to convergence against continuous test functions. For an interval \(I\subset\mathbb R\), \(\LL^1\!\llcorner I\) denotes Lebesgue measure restricted to \(I\) (unnormalised). We write \(\mathrm{pr}_1,\mathrm{pr}_2\) for the projections of a product space onto its two factors, \(\pi_i(x_1,x_2):=x_i\) for the coordinate projections on \(\mathbb R^2\), and
\[
P:K\times K\to\mathbb R\times\mathbb R,
\qquad
P(x,y):=(x_1,y_1),
\]
for the horizontal projection. The fibre disintegrations of \(\nu_n\) and \(\nu\) are
\begin{equation}\label{eq:fibre-disintegration-n}
\nu_n=\int_0^1 \nu_n^t\,dt,
\qquad
\nu_n^t:=
\begin{cases}
(\LL^1\!\llcorner[0,1])\otimes\delta_t, & t\in A_n,\\
(\LL^1\!\llcorner[-1,0])\otimes\delta_t, & t\in B_n,
\end{cases}
\end{equation}
and
\begin{equation}\label{eq:fibre-disintegration-limit}
\nu=\int_0^1 \nu^t\,dt,
\qquad
\nu^t:=\tfrac12\,(\LL^1\!\llcorner[-1,1])\otimes\delta_t.
\end{equation}

\subsection{A common Kantorovich potential and its contact set}\label{sec:counterexample-contact}
For \(c(x,y)=|x-y|\), Kantorovich duality on the compact metric space \(K\) is given by
\[
m_1(\mu,\nu)=\max\Bigl\{\int_K u\,d\mu-\int_K u\,d\nu:\ u\in\Lip(K),\ \Lip(u)\le1\Bigr\}.
\]
Given a \(1\)-Lipschitz \(u:K\to\mathbb R\), let us define its \textit{contact set}
\[
\Sigma_u:=\{(x,y)\in K\times K:\ u(x)-u(y)=|x-y|\}.
\]

\begin{lemma}[Contact set characterisation of \(W_1\)-optimal plans]\label{lem:contact-general}
Let \(\mu,\nu\in\Pcal(K)\) and let \(u\) be an optimal \(1\)-Lipschitz Kantorovich potential for \((\mu,\nu)\). It then holds
\[
m_1(\mu,\nu)=\int_K u\,d\mu-\int_K u\,d\nu,
\qquad
O(\mu,\nu)=\{\eta\in\Pi(\mu,\nu):\eta(\Sigma_u)=1\}.
\]
\end{lemma}

\begin{proof}
For an arbitrary \(\eta\in\Pi(\mu,\nu)\),
\[
\int |x-y|\,d\eta \ge \int (u(x)-u(y))\,d\eta
=\int u\,d\mu-\int u\,d\nu,
\]
with equality if and only if \(\eta\) is supported on \(\Sigma_u\).
Given that \(u\) is optimal, the right-hand side equals \(m_1(\mu,\nu)\), so the characterisation follows immediately.
\end{proof}

Let us define
\begin{equation}\label{eq:Sigma}
\Sigma:=\{(x,y)\in K\times K:\ x_2=y_2,\ y_1\le x_1\}.
\end{equation}

\begin{lemma}[Contact-set description in the laminate family]\label{lem:contact-set}
The function \(u(x)=x_1\) is an optimal \(1\)-Lipschitz Kantorovich potential both for \((\mu,\nu_n)\) and for \((\mu,\nu)\), and
\[
m_1(\mu,\nu_n)=m_1(\mu,\nu)=\frac12.
\]
Moreover, \(\Sigma_u=\Sigma\), where \(\Sigma\) is given by \eqref{eq:Sigma}.
As a result,
\[
O(\mu,\nu_n)=\{\eta\in\Pi(\mu,\nu_n):\eta(\Sigma)=1\},
\qquad
O(\mu,\nu)=\{\eta\in\Pi(\mu,\nu):\eta(\Sigma)=1\}.
\]
In particular, every \(\eta\in O(\mu,\nu_n)\cup O(\mu,\nu)\) satisfies \(x_2=y_2\) \(\eta\)-a.e.
\end{lemma}

\begin{proof}
Since \(u(x)=x_1\) is \(1\)-Lipschitz, it is admissible in the dual problem.
By direct computation, we obtain
\[
\int u\,d\mu=\int_{[0,1]^2}x_1\,dx_1dx_2=\frac12,
\qquad
\int u\,d\nu_n=0=\int u\,d\nu,
\]
so the dual value equals \(1/2\) for both \((\mu,\nu_n)\) and \((\mu,\nu)\).
On the other hand, we have
\[
C_1(\gamma_n)=\int_{[0,1]^2}|x-T_n(x)|\,dx=\LL^1(B_n)=\frac12,
\qquad
C_1(\gamma^{\mathrm{hom}})=\frac12,
\]
so \(m_1(\mu,\nu_n)=m_1(\mu,\nu)=1/2\), and \(u\) is optimal.
Equality in \(|x-y|\ge u(x)-u(y)=x_1-y_1\) holds if and only if \(x_2=y_2\) and \(y_1\le x_1\), so \(\Sigma_u=\Sigma\). The characterisation of \(O(\mu,\nu_n)\) and \(O(\mu,\nu)\) follows directly from Lemma~\ref{lem:contact-general}.
\end{proof}

\subsection{Fibrewise reduction}\label{sec:fiberwise-reduction}
\begin{lemma}[Fibrewise disintegration along horizontal slices]\label{lem:fiberwise-reduction}
Recall the contact set \(\Sigma\) from \eqref{eq:Sigma}.
Let \(\widetilde\nu\in\Pcal(K)\) satisfy \((\pi_2)_\#\widetilde\nu=\LL^1\!\llcorner[0,1]\).
Let \(\eta\in\Pi(\mu,\widetilde\nu)\) with \(\eta(\Sigma)=1\). It follows that \(x_2=y_2\) \(\eta\)-a.e., and \(\eta\) admits a disintegration with respect to the Borel map
\[
\tau:K\times K\to[0,1],\qquad \tau(x,y):=x_2,
\]
of the form
\[
\eta=\int_0^1 \eta^t\,dt,
\qquad
\eta^t(\{x_2=y_2=t\})=1\ \text{ for a.e.\ }t\in[0,1].
\]
We now set \(\bar\eta^t:=P_\#\eta^t\). For a.e.\ \(t\in[0,1]\), we obtain
\[
(\mathrm{pr}_1)_\#\bar\eta^t=\LL^1\!\llcorner[0,1],
\qquad
\bar\eta^t(\{(x,y):y\le x\})=1,
\]
Moreover, suppose that
\[
J_t:[0,1]\times[-1,1]\to K\times K,
\qquad
J_t(x,y):=((x,t),(y,t)),
\]
then
\[
\eta^t=(J_t)_\#\bar\eta^t
\qquad\text{for a.e.\ }t\in[0,1].
\]
In particular, the family \((\bar\eta^t)\) uniquely determines \(\eta\).
Finally, we fix a disintegration
\[
\widetilde\nu=\int_0^1\widetilde\nu^t\,dt,
\qquad
\widetilde\nu^t(\{y_2=t\})=1\ \text{for a.e.\ }t\in[0,1],
\]
After modification on a null set of \(t\), we arrive at
\[
(\mathrm{pr}_2)_\#\bar\eta^t=(\pi_1)_\#\widetilde\nu^t
\]
for a.e.\ \(t\in[0,1]\).
\end{lemma}

\begin{proof}
From \(\eta(\Sigma)=1\) and \(\Sigma=\{x_2=y_2,\ y_1\le x_1\}\), we have
\[
x_2=y_2,\qquad y_1\le x_1
\qquad \eta\text{-a.e.}
\]
We now disintegrate \(\eta\) with respect to
\[
\tau:K\times K\to[0,1],\qquad \tau(x,y):=x_2.
\]
Because \((\mathrm{pr}_1)_\#\eta=\mu\) holds, we have
\[
\tau_\#\eta=(\pi_2)_\#\mu=\LL^1\!\llcorner[0,1].
\]
Finally, the disintegration theorem gives a measurable family \((\eta^t)_{t\in[0,1]}\) such that
\[
\eta=\int_0^1\eta^t\,dt,
\qquad
\eta^t(\{x_2=t\})=1
\quad\text{for a.e.\ }t\in[0,1].
\]
Since \(x_2=y_2\) \(\eta\)-a.e., after the modification on a null set, we may also assume
\[
\eta^t(\{x_2=y_2=t\})=1
\quad\text{for a.e.\ }t\in[0,1].
\]

Let us now set \(\bar\eta^t:=P_\#\eta^t\). Since \((\mathrm{pr}_1)_\#\eta=\mu\) and \(\eta=\int_0^1\eta^t\,dt\), we have
\[
\mu=\int_0^1 (\mathrm{pr}_1)_\#\eta^t\,dt,
\qquad
(\mathrm{pr}_1)_\#\eta^t(\{x_2=t\})=1
\quad\text{for a.e.\ }t\in[0,1].
\]
We infer thats the family \(\bigl((\mathrm{pr}_1)_\#\eta^t\bigr)_{t\in[0,1]}\) is a disintegration of \(\mu\) with respect to \(\pi_2\). By uniqueness of disintegration, after modification on a null set, we may assume
\[
(\mathrm{pr}_1)_\#\eta^t
=
(\LL^1\!\llcorner[0,1])\otimes\delta_t
\quad\text{for a.e.\ }t\in[0,1].
\]
We apply \(\pi_1\) and obtain
\[
(\mathrm{pr}_1)_\#\bar\eta^t=\LL^1\!\llcorner[0,1]
\quad\text{for a.e.\ }t\in[0,1].
\]
Analogously, since \(y_1\le x_1\) \(\eta\)-a.e., we have
\[
\bar\eta^t(\{(x,y):y\le x\})=1
\quad\text{for a.e.\ }t\in[0,1].
\]
For a.e.\ \(t\), the measure \(\eta^t\) is supported on \(\{x_2=y_2=t\}\), and \(J_t\circ P=\id\) holds on this set. Therefore,
\[
\eta^t=(J_t)_\#\bar\eta^t
\quad\text{for a.e.\ }t.
\]
Finally, let \(\nu_\eta^t:=(\mathrm{pr}_2)_\#\eta^t\). We obtain
\[
\widetilde\nu=\int_0^1 \nu_\eta^t\,dt,
\qquad
\nu_\eta^t(\{y_2=t\})=1
\quad\text{for a.e.\ }t,
\]
so \((\nu_\eta^t)\) is a disintegration of \(\widetilde\nu\) with respect to \(y_2\). Moreover,
\[
(\mathrm{pr}_2)_\#\bar\eta^t=(\pi_1)_\#\nu_\eta^t
\quad\text{for a.e.\ }t.
\]
Let us fix a disintegration
\[
\widetilde\nu=\int_0^1\widetilde\nu^t\,dt,
\qquad
\widetilde\nu^t(\{y_2=t\})=1
\quad\text{for a.e.\ }t.
\]
By uniqueness of disintegration up to null sets, after modification on a null set, we may assume \(\widetilde\nu^t=\nu_\eta^t\) for a.e.\ \(t\). Finally,
\[
(\mathrm{pr}_2)_\#\bar\eta^t=(\pi_1)_\#\widetilde\nu^t
\quad\text{for a.e.\ }t.
\]
\end{proof}

\begin{corollary}[Fibrewise integral and cost decomposition]\label{cor:fiberwise-decomposition}
Under the hypotheses and notation of Lemma~\ref{lem:fiberwise-reduction}, for every Borel \(\psi:\mathbb R\times\mathbb R\to\mathbb R\) such that
\((x,y)\mapsto\psi(x_1,y_1)\) is \(\eta\)-integrable, it holds that
\[
\int_{K\times K}\psi(x_1,y_1)\,d\eta(x,y)
=\int_0^1\int_{\mathbb R\times\mathbb R}\psi(x,y)\,d\bar\eta^t(x,y)\,dt.
\]
In particular, for every \(p\ge1\),
\[
C_p(\eta)=\int_0^1\int_{\mathbb R\times\mathbb R}|x-y|^p\,d\bar\eta^t(x,y)\,dt.
\]
\end{corollary}

\begin{proof}
Let \((\eta^t)\) and \((\bar\eta^t)\) be the disintegration and fibrewise pushforwards from Lemma~\ref{lem:fiberwise-reduction}. For Borel \(\psi\) as above, we define \(F(x,y):=\psi(x_1,y_1)\). It follows that \(F\in L^1(\eta)\), and disintegration gives
\[
\int_{K\times K}F\,d\eta=\int_0^1\int_{K\times K}F\,d\eta^t\,dt.
\]
We apply \(\bar\eta^t=P_\#\eta^t\), and obtain
\[
\int_{K\times K}\psi(x_1,y_1)\,d\eta
=\int_0^1\int_{\mathbb R\times\mathbb R}\psi(x,y)\,d\bar\eta^t(x,y)\,dt.
\]
If we now take \(\psi(x,y)=|x-y|^p\) and use \(x_2=y_2\), \(\eta\)-a.e., we arrive at
\[
C_p(\eta)=\int|x-y|^p\,d\eta=\int|x_1-y_1|^p\,d\eta
=\int_0^1\int_{\mathbb R\times\mathbb R}|x-y|^p\,d\bar\eta^t\,dt.
\]
\end{proof}

\begin{lemma}[Fibrewise description of \(O(\mu,\nu_n)\) and \(O(\mu,\nu)\)]\label{lem:fiberwise-admissible}
Let us now consider the following cases.
\begin{enumerate}[(i)]
\item If \(n\in\mathbb N\) and \(\eta\in O(\mu,\nu_n)\), then \(\eta=\int_0^1\eta^t\,dt\), where
\[
\eta^t(\{x_2=y_2=t\})=1
\quad\text{for a.e.\ }t\in[0,1],
\]
and the fibre couplings \(\bar\eta^t:=P_\#\eta^t\) satisfy
\[
(\mathrm{pr}_1)_\#\bar\eta^t=\LL^1\!\llcorner[0,1],
\qquad
\bar\eta^t(\{y\le x\})=1
\]
for a.e.\ \(t\in[0,1]\), together with
\[
(\mathrm{pr}_2)_\#\bar\eta^t=
\begin{cases}
\LL^1\!\llcorner[0,1], & t\in A_n,\\
\LL^1\!\llcorner[-1,0], & t\in B_n,
\end{cases}
\qquad\text{for a.e.\ }t\in[0,1].
\]

\item If \(\eta\in O(\mu,\nu)\), then \(\eta=\int_0^1\eta^t\,dt\), where
\[
\eta^t(\{x_2=y_2=t\})=1
\quad\text{for a.e.\ }t\in[0,1],
\]
and the fibre couplings \(\bar\eta^t:=P_\#\eta^t\) satisfy
\[
(\mathrm{pr}_1)_\#\bar\eta^t=\LL^1\!\llcorner[0,1],
\qquad
(\mathrm{pr}_2)_\#\bar\eta^t=\tfrac12\,\LL^1\!\llcorner[-1,1],
\qquad
\bar\eta^t(\{y\le x\})=1
\]
for a.e.\ \(t\in[0,1]\).
\end{enumerate}
\end{lemma}

\begin{proof}
If \(\eta\in O(\mu,\nu_n)\) or \(\eta\in O(\mu,\nu)\), then Lemma~\ref{lem:contact-set} gives \(\eta(\Sigma)=1\). These conclusions follow from the results of Lemma~\ref{lem:fiberwise-reduction}, Corollary~\ref{cor:fiberwise-decomposition}, and the disintegrations \eqref{eq:fibre-disintegration-n} and \eqref{eq:fibre-disintegration-limit}.
\end{proof}

\subsection{One-dimensional minimisers}\label{subsec:prelim-fiber}
\begin{lemma}[Fibrewise order rigidity with equal marginals]\label{lem:fiber-order-rigidity}
Let \(\pi\in\Pcal(\mathbb R\times\mathbb R)\) satisfy \((\mathrm{pr}_1)_\#\pi=(\mathrm{pr}_2)_\#\pi=\LL^1\!\llcorner[0,1]\) and \(\pi(\{y\le x\})=1\).
It then follows that \(\pi=(\id,\id)_\#(\LL^1\!\llcorner[0,1])\). In other terms, \(\pi(\{(x,y):y=x\})=1\).
\end{lemma}

\begin{proof}
Since \(\pi(\{y\le x\})=1\), the function \(x-y\) is nonnegative \(\pi\)-a.e. Moreover,
\[
\int_{\mathbb R\times\mathbb R}(x-y)\,d\pi(x,y)
=\int_0^1 x\,dx-\int_0^1 y\,dy
=0.
\]
The above implies that \(x-y=0\) \(\pi\)-a.e., so \(\pi\) is concentrated on the diagonal. Since both marginals are \(\LL^1\!\llcorner[0,1]\), it follows that
\[
\pi=(\id,\id)_\#(\LL^1\!\llcorner[0,1]).
\]
\end{proof}

\begin{lemma}[One-dimensional quadratic minimisers]\label{lem:fiber-quadratic}
Let \(\pi\in\Pcal(\mathbb R\times\mathbb R)\).
\begin{enumerate}[(i)]
\item If \((\mathrm{pr}_1)_\#\pi=\LL^1\!\llcorner[0,1]\) and \((\mathrm{pr}_2)_\#\pi=\LL^1\!\llcorner[-1,0]\), then
\[
\int (x-y)^2\,d\pi\ge 1,
\]
with equality if and only if \(\pi(\{(x,y):y=x-1\})=1\).
\item If \((\mathrm{pr}_1)_\#\pi=\LL^1\!\llcorner[0,1]\) and \((\mathrm{pr}_2)_\#\pi=\tfrac12\,\LL^1\!\llcorner[-1,1]\), then
\[
\int (x-y)^2\,d\pi\ge \frac13,
\]
with equality if and only if \(\pi(\{(x,y):y=2x-1\})=1\).
\end{enumerate}
\end{lemma}

\begin{proof}
\begin{enumerate}[(i)]
\item Since
\[
\int_{\mathbb R\times\mathbb R}(x-y)\,d\pi(x,y)
=\int_0^1 x\,dx-\int_{-1}^0 y\,dy
=1,
\]
By Jensen's inequality, we obtain
\[
\int (x-y)^2\,d\pi \ge \left(\int (x-y)\,d\pi\right)^2 = 1.
\]
Equality holds if and only if \(x-y\) is \(\pi\)-a.e.\ constant, that is equal to \(1\). In other words, this holds if and only if \(y=x-1\) \(\pi\)-a.e.

\item Note that the marginals are fixed, so
\[
\int (x-y)^2\,d\pi
\;=\;
\int_0^1 x^2\,dx
+
\int_{-1}^1 \frac12\,y^2\,dy
-
2\int xy\,d\pi.
\]
This implies that minimising \(\int (x-y)^2\,d\pi\) is equivalent to maximising \(\int xy\,d\pi\). By the classical one-dimensional monotone rearrangement theorem for strictly convex costs, the unique minimiser is induced by the monotone map \(T=F_Y^{[-1]}\circ F_X\) (see \cite[Theorems~2.5 and~2.9]{Santambrogio2015}), where \(F_X\) and \(F_Y\) denote the distribution functions of the first and second marginals and
\[
F_Y^{[-1]}(s):=\inf\{t\in\mathbb R:\ F_Y(t)\ge s\}
\]
is the monotone generalised inverse. Here, \(F_X(x)=x\) on \([0,1]\) and \(F_Y^{[-1]}(s)=2s-1\) on \([0,1]\), so \(T(x)=2x-1\). It follows that the unique minimiser is
\[
\left(\id,2\id-1\right)_\#(\LL^1\!\llcorner[0,1]),
\]
and the minimum equals
\[
\int_0^1 (x-(2x-1))^2\,dx=\frac13.
\]
\end{enumerate}
\end{proof}

\begin{lemma}[Fibrewise strictly convex minimum for shifted marginals]\label{lem:fiber-strict-convex}
Let \(\Phi:[0,2]\to\mathbb R\) be strictly convex, and let \(\pi\in\Pcal(\mathbb R\times\mathbb R)\) satisfy \((\mathrm{pr}_1)_\#\pi=\LL^1\!\llcorner[0,1]\) and \((\mathrm{pr}_2)_\#\pi=\LL^1\!\llcorner[-1,0]\).
It follows that
\[
\int \Phi(|x-y|)\,d\pi \ge \Phi(1),
\]
with equality if and only if \(\pi(\{(x,y):y=x-1\})=1\).
\end{lemma}

\begin{proof}
Since \(x\in[0,1]\) and \(y\in[-1,0]\) \(\pi\)-a.e., \(|x-y|=x-y\) \(\pi\)-a.e. follows. Moreover,
\[
\int (x-y)\,d\pi(x,y)
=\int_0^1 x\,dx-\int_{-1}^0 y\,dy
=1.
\]
By Jensen's inequality,
\[
\int \Phi(|x-y|)\,d\pi
=
\int \Phi(x-y)\,d\pi
\ge
\Phi\!\left(\int (x-y)\,d\pi\right)
=
\Phi(1).
\]
Given that \(\Phi\) is strictly convex, equality holds if and only if \(x-y=1\) \(\pi\)-a.e., that is, if and only if \(y=x-1\) \(\pi\)-a.e.
\end{proof}

\section{Counterexample to narrow stability of the ray-monotone selector}\label{sec:counterexample}

We note that by Lemma~\ref{lem:contact-set}, every plan in \(O(\mu,\nu_n)\cup O(\mu,\nu)\) is supported on the common contact set \(\Sigma\), so the analysis in this section can be reduced to one-dimensional fibres.

\subsection{Convergence of the oscillatory targets and plans}
\begin{lemma}[Convergence of \(\nu_n\) and \(\gamma_n\)]\label{lem:conv-nu-gamma}
With \(\nu_n,\nu,\gamma_n,\gamma^{\mathrm{hom}}\) as in Definition~\ref{def:intro-osc}, we have
\[
\nu_n\rightharpoonup \nu \quad\text{in }\Pcal(K),
\qquad
\gamma_n\rightharpoonup \gamma^{\mathrm{hom}} \quad\text{in }\Pcal(K\times K).
\]
\end{lemma}

\begin{proof}
Let \(\varphi\in C(K)\). We apply \(\nu_n=(T_n)_\#\mu\) and Fubini's theorem,
\[
\int_K \varphi(y)\,d\nu_n(y)
=\int_{[0,1]^2}\varphi(T_n(x_1,x_2))\,dx_1dx_2
=\int_{A_n} f_1(t)\,dt + \int_{B_n} f_2(t)\,dt,
\]
where
\[
f_1(t):=\int_0^1 \varphi(x_1,t)\,dx_1,
\qquad
f_2(t):=\int_0^1 \varphi(x_1-1,t)\,dx_1.
\]
Since \(f_1,f_2\in L^1([0,1])\), Lemma~\ref{lem:weakstar-An} yields
\[
\int_K \varphi\,d\nu_n \to \frac12\int_0^1 f_1(t)\,dt+\frac12\int_0^1 f_2(t)\,dt = \int_K \varphi\,d\nu,
\]
which proves \(\nu_n\rightharpoonup \nu\).

Analogously, let \(\Psi\in C(K\times K)\). We apply \(\gamma_n=(\id\times T_n)_\#\mu\),
\[
\int_{K\times K} \Psi(x,y)\,d\gamma_n(x,y)
=\int_{[0,1]^2}\Psi\bigl((x_1,x_2),T_n(x_1,x_2)\bigr)\,dx_1dx_2
=\int_{A_n} g_1(t)\,dt+\int_{B_n} g_2(t)\,dt,
\]
where
\[
g_1(t):=\int_0^1 \Psi((x_1,t),(x_1,t))\,dx_1,
\qquad
g_2(t):=\int_0^1 \Psi((x_1,t),(x_1-1,t))\,dx_1.
\]
Again \(g_1,g_2\in L^1([0,1])\), so Lemma~\ref{lem:weakstar-An} gives
\[
\int \Psi\,d\gamma_n \to \frac12\int \Psi(x,x)\,d\mu(x)+\frac12\int \Psi(x,Sx)\,d\mu(x) = \int \Psi\,d\gamma^{\mathrm{hom}},
\]
that is,\ \(\gamma_n\rightharpoonup\gamma^{\mathrm{hom}}\).
\end{proof}

\subsection{The oscillatory problems}

\begin{proposition}[Quadratic minimiser for the oscillatory targets]\label{prop:sec-n}
For each \(n\), the plan \(\gamma_n\) is the unique minimiser of \(C_2\) over \(O(\mu,\nu_n)\).
\end{proposition}

\begin{proof}
Let us fix \(n\) and \(\eta\in O(\mu,\nu_n)\). By Lemma~\ref{lem:fiberwise-admissible}(i), the fibre couplings \(\bar\eta^t\) satisfy
\[
(\mathrm{pr}_1)_\#\bar\eta^t=\LL^1\!\llcorner[0,1],\qquad
\bar\eta^t(\{y\le x\})=1
\]
for a.e.\ \(t\in[0,1]\), with
\[
(\mathrm{pr}_2)_\#\bar\eta^t=
\begin{cases}
\LL^1\!\llcorner[0,1], & t\in A_n,\\
\LL^1\!\llcorner[-1,0], & t\in B_n.
\end{cases}
\]
By Lemma~\ref{lem:fiber-order-rigidity}, we obtain
\[
\bar\eta^t=(\id,\id)_\#(\LL^1\!\llcorner[0,1])
\quad\text{for a.e.\ }t\in A_n,
\]
while Lemma~\ref{lem:fiber-quadratic}(i) gives
\[
\int (x-y)^2\,d\bar\eta^t(x,y)\ge 1
\quad\text{for a.e.\ }t\in B_n,
\]
with equality if and only if
\[
\bar\eta^t=(\id,\id-1)_\#(\LL^1\!\llcorner[0,1]).
\]
Corollary~\ref{cor:fiberwise-decomposition} gives
\[
C_2(\eta)\ge \int_{B_n}1\,dt=\frac12.
\]
Since \(C_2(\gamma_n)=\frac12\), the plan \(\gamma_n\) is minimal. If equality holds, then the equality cases above hold on almost every fibre, and the reconstruction formula in Lemma~\ref{lem:fiberwise-reduction} yields \(\eta=\gamma_n\). We conclude that \(\gamma_n\) is the unique minimiser.
\end{proof}

\subsection{The limit problem}
\begin{proposition}[Quadratic minimiser for the homogenised target]\label{prop:sec-limit}
The functional \(C_2\) has a unique minimiser on \(O(\mu,\nu)\), namely
\[
\gamma^{\mathrm{sel}}(\mu,\nu)=(\id\times T)_\#\mu,
\qquad
T(x_1,x_2):=(2x_1-1,x_2).
\]
In particular, \(C_2\bigl(\gamma^{\mathrm{sel}}(\mu,\nu)\bigr)=\frac13\).
\end{proposition}

\begin{proof}
Let \(\eta\in O(\mu,\nu)\). By Lemma~\ref{lem:fiberwise-admissible}(ii), the fibre couplings \(\bar\eta^t\) have first marginal \(\LL^1\!\llcorner[0,1]\), second marginal \(\tfrac12\,\LL^1\!\llcorner[-1,1]\), and are supported on \(\{y\le x\}\) for a.e.\ \(t\in[0,1]\). Lemma~\ref{lem:fiber-quadratic}(ii) gives
\[
\int (x-y)^2\,d\bar\eta^t(x,y)\ge \frac13
\quad\text{for a.e.\ }t,
\]
with equality if and only if
\[
\bar\eta^t=(\id,2\id-1)_\#(\LL^1\!\llcorner[0,1]).
\]
By Corollary~\ref{cor:fiberwise-decomposition}, we have
\[
C_2(\eta)\ge \frac13.
\]

We set \(T(x_1,x_2):=(2x_1-1,x_2)\). Since \(T_\#\mu=\nu\) and \((\id\times T)_\#\mu\) is supported on \(\Sigma\), Lemma~\ref{lem:contact-set} gives \((\id\times T)_\#\mu\in O(\mu,\nu)\). Moreover,
\[
C_2\bigl((\id\times T)_\#\mu\bigr)
=\int_{[0,1]^2}|(x_1,x_2)-(2x_1-1,x_2)|^2\,dx_1dx_2
=\int_0^1(1-x_1)^2\,dx_1
=\frac13.
\]
Thus \((\id\times T)_\#\mu\) is a minimiser. If equality holds for \(\eta\), then the equality case in Lemma~\ref{lem:fiber-quadratic}(ii) holds on almost every fibre, and the reconstruction formula in Lemma~\ref{lem:fiberwise-reduction} yields
\[
\eta=(\id\times T)_\#\mu.
\]
This implies that the minimiser is unique, that is,
\[
\gamma^{\mathrm{sel}}(\mu,\nu)=(\id\times T)_\#\mu.
\]
Since \(T(x)=x\) if and only if \(x_1=1\), we also have
\[
\gamma^{\mathrm{sel}}(\mu,\nu)(\Delta)=\mu(\{x_1=1\})=0.
\]
\end{proof}

\subsection{Proof of the instability theorem}
\begin{proof}[Proof of Theorem~\ref{thm:main}]
Lemma~\ref{lem:conv-nu-gamma} gives \(\nu_n\rightharpoonup\nu\) and \(\gamma_n\rightharpoonup\gamma^{\mathrm{hom}}\). By Proposition~\ref{prop:sec-n}, \(\gamma_n=\gamma^{\mathrm{sel}}(\mu,\nu_n)\) holds for every \(n\). Proposition~\ref{prop:sec-limit} yields
\[
\gamma^{\mathrm{sel}}(\mu,\nu)=(\id\times T)_\#\mu,
\qquad
T(x_1,x_2)=(2x_1-1,x_2),
\]
together with
\[
C_2\bigl(\gamma^{\mathrm{sel}}(\mu,\nu)\bigr)=\frac13,
\qquad
\gamma^{\mathrm{sel}}(\mu,\nu)(\Delta)=0.
\]

We shall now identify the limit plan. By construction,
\[
(\mathrm{pr}_1)_\#\gamma^{\mathrm{hom}}=\mu,
\qquad
(\mathrm{pr}_2)_\#\gamma^{\mathrm{hom}}=\nu,
\]
and both \((\id,\id)_\#\mu\) and \((\id,S)_\#\mu\) are supported on \(\Sigma\). It follows that \(\gamma^{\mathrm{hom}}\in O(\mu,\nu)\) by Lemma~\ref{lem:contact-set}. Moreover,
\[
\gamma^{\mathrm{hom}}(\Delta)
=
\frac12(\id,\id)_\#\mu(\Delta)+\frac12(\id,S)_\#\mu(\Delta)
=
\frac12,
\]
because \(S\) has no fixed point on \([0,1]^2\). Finally,
\[
C_2(\gamma^{\mathrm{hom}})
=
\frac12 C_2\bigl((\id,\id)_\#\mu\bigr)
+
\frac12 C_2\bigl((\id,S)_\#\mu\bigr)
=
\frac12.
\]
We conclude that \(\gamma^{\mathrm{hom}}\neq\gamma^{\mathrm{sel}}(\mu,\nu)\), which establishes the claim.
\end{proof}

\subsection{Kuratowski limit of the optimal-plan sets}\label{sec:stable}

We now identify the exact class of plans in \(O(\mu,\nu)\) that arise as narrow limits of plans in \(O(\mu,\nu_n)\).

\begin{definition}[Kuratowski upper and lower limits]\label{def:stable-closure}
For subsets \(E_n\subset\Pcal(K\times K)\), let us define
\[
\limsup_{n\to\infty}E_n
:=\Bigl\{\eta:\exists\,n_k\uparrow\infty,\ \eta_{n_k}\in E_{n_k},\ \eta_{n_k}\rightharpoonup\eta\Bigr\},
\]
\[
\liminf_{n\to\infty}E_n
:=\Bigl\{\eta:\forall\ \text{open neighbourhoods }U\ni\eta\text{ in the narrow topology},\ \exists N\ \text{s.t.}\ E_n\cap U\neq\varnothing\ \forall n\ge N\Bigr\}.
\]
If these two sets coincide, we denote their common value by \(\lim_{n\to\infty}E_n\).
\end{definition}

\begin{lemma}[Equivalent characterisation of \(\Aatt\)]\label{lem:Delta-equiv}
For \(\eta\in O(\mu,\nu)\), the following are equivalent:
\begin{enumerate}[(i)]
\item \(\eta(\Delta)=\frac12\)
\item \(\eta\llcorner\{y_1\ge0\}=\frac12(\id,\id)_\#\mu\).
\end{enumerate}
\end{lemma}

\begin{proof}
We assume that \(\eta\in O(\mu,\nu)\) and \textit{(i)} holds.
Since \((\mathrm{pr}_1)_\#\eta=\mu\) and \(\mu(\{x_1<0\})=0\), we have
\[
\eta(\Delta\cap\{y_1<0\})
=\eta(\{x=y,\ x_1<0\})
\le \eta(\{x_1<0\})
=\mu(\{x_1<0\})
=0,
\]
which implies \(\eta(\Delta\cap\{y_1\ge0\})=\eta(\Delta)=\frac12\).
From \(\eta(\{y_1\ge0\})=\nu(\{y_1\ge0\})=\frac12\), it follows that \(\eta(\{y_1\ge0\}\setminus\Delta)=0\), that is, \(\eta\llcorner\{y_1\ge0\}\) is supported on \(\Delta\).
Any finite measure supported on \(\Delta\) has the form \((\id,\id)_\#\alpha\) for \(\alpha=(\mathrm{pr}_1)_\#\zeta=(\mathrm{pr}_2)_\#\zeta\).
We now apply this to \(\zeta:=\eta\llcorner\{y_1\ge0\}\) and obtain
\[
\eta\llcorner\{y_1\ge0\}=(\id,\id)_\#\alpha,
\qquad
\alpha=(\mathrm{pr}_2)_\#(\eta\llcorner\{y_1\ge0\})
=\nu\llcorner\{y_1\ge0\}=\frac12\,\mu,
\]
so \(\eta\llcorner\{y_1\ge0\}=\frac12(\id,\id)_\#\mu\) follows.

Conversely, let us assume that \textit{(ii)} holds. It follows that
\(\eta(\Delta\cap\{y_1\ge0\})=\eta(\{y_1\ge0\})=\frac12\) because \((\id,\id)_\#\mu\) is supported on \(\Delta\). Moreover, \((\mathrm{pr}_1)_\#\eta=\mu\) and \(\mu(\{x_1<0\})=0\) imply
\[
\eta(\Delta\cap\{y_1<0\})
=\eta(\{x=y,\ x_1<0\})
\le \eta(\{x_1<0\})
=0.
\]
It follows that \(\eta(\Delta)=\frac12\).
\end{proof}

\begin{lemma}[Decomposition of attainable plans]\label{lem:Aatt-decomp}
Let \(\eta\in\Aatt\), and define
\[
\widehat\eta:=2\,\eta\llcorner\{y_1<0\}.
\]
We then have \(\widehat\eta\in\Pcal(K\times K)\) and
\[
\eta=\frac12(\id,\id)_\#\mu+\frac12\widehat\eta,
\]
\[
(\mathrm{pr}_1)_\#\widehat\eta=\mu,
\qquad
(\mathrm{pr}_2)_\#\widehat\eta=S_\#\mu,
\qquad
\widehat\eta(\Sigma)=1.
\]
\end{lemma}

\begin{proof}
Since \((\mathrm{pr}_2)_\#\eta=\nu=\frac12\mu+\frac12 S_\#\mu\), we have
\[
\eta(\{y_1<0\})=\nu(\{y_1<0\})=\frac12,
\]
so \(\widehat\eta\) is a probability measure. By Lemma~\ref{lem:Delta-equiv},
\[
\eta\llcorner\{y_1\ge0\}=\frac12(\id,\id)_\#\mu.
\]
We then have,
\[
\eta=\eta\llcorner\{y_1\ge0\}+\eta\llcorner\{y_1<0\}
=\frac12(\id,\id)_\#\mu+\frac12\widehat\eta.
\]
Because \(\eta\in O(\mu,\nu)\), Lemma~\ref{lem:contact-set} gives \(\eta(\Sigma)=1\), so \(\widehat\eta(\Sigma)=1\) follows.

For the first marginal,
\[
(\mathrm{pr}_1)_\#\widehat\eta
=2(\mathrm{pr}_1)_\#(\eta\llcorner\{y_1<0\})
=2\Bigl(\mu-(\mathrm{pr}_1)_\#(\eta\llcorner\{y_1\ge0\})\Bigr)
=2\Bigl(\mu-\frac12\mu\Bigr)
=\mu.
\]
Analogously,
\[
(\mathrm{pr}_2)_\#\widehat\eta
=2(\mathrm{pr}_2)_\#(\eta\llcorner\{y_1<0\})
=2\Bigl(\nu-(\mathrm{pr}_2)_\#(\eta\llcorner\{y_1\ge0\})\Bigr)
=2\Bigl(\nu-\frac12\mu\Bigr)
=S_\#\mu.
\]
\end{proof}

\begin{proposition}[Rigidity and limsup inclusion for \(O(\mu,\nu_n)\)]\label{prop:stable-limsup}
Let \(\eta_n\in O(\mu,\nu_n)\). It holds that \(\eta_n(\Delta)=\frac12\) for every \(n\), and any narrow limit point \(\eta\) of \((\eta_n)\) satisfies \(\eta\in\Aatt\).
\end{proposition}

\begin{proof}
We fix \(n\) and \(\eta_n\in O(\mu,\nu_n)\). By Lemma~\ref{lem:fiberwise-admissible}(i), the fibre couplings over \(A_n\) have equal marginals \(\LL^1\!\llcorner[0,1]\) and are supported on \(\{y\le x\}\). By Lemma~\ref{lem:fiber-order-rigidity}, we obtain
\[
\bar\eta_n^t=(\id,\id)_\#(\LL^1\!\llcorner[0,1])
\quad\text{for a.e.\ }t\in A_n.
\]
It implies that these fibres contribute diagonal mass \(1\). For a.e.\ \(t\in B_n\), the first marginal is supported on \([0,1]\) and the second on \([-1,0]\), so diagonal mass can occur only at \((0,0)\). Note that both marginals are nonatomic, so \(\bar\eta_n^t(\{(0,0)\})=0\) holds. Therefore,
\[
\eta_n(\Delta)=\int_{A_n}1\,dt+\int_{B_n}0\,dt=\frac12.
\]
Let \(\eta_n\rightharpoonup\eta\) along a subsequence. By the continuity of the projections,
\[
(\mathrm{pr}_1)_\#\eta=\mu,\qquad
(\mathrm{pr}_2)_\#\eta=\nu.
\]
Since \(\Sigma\) is closed and \(\eta_n(\Sigma)=1\), Portmanteau gives \(\eta(\Sigma)=1\), so \(\eta\in O(\mu,\nu)\) by Lemma~\ref{lem:contact-set}. Analogously, the closedness of \(\Delta\) gives
\[
\eta(\Delta)\ge \limsup_{n\to\infty}\eta_n(\Delta)=\frac12.
\]
On the other hand, \((\mathrm{pr}_1)_\#\eta=\mu\) and \(\mu(\{x_1<0\})=0\) imply
\[
\eta(\Delta\cap\{y_1<0\})=\eta(\{x=y,\ x_1<0\})=0,
\]
so
\[
\eta(\Delta)=\eta(\Delta\cap\{y_1\ge0\})\le \eta(\{y_1\ge0\})=\nu(\{y_1\ge0\})=\frac12.
\]
We conclude that \(\eta(\Delta)=\frac12\), and thus \(\eta\in\Aatt\).
\end{proof}

\begin{proposition}[Recovery and liminf inclusion for \(\Aatt\)]\label{prop:stable-liminf}
For every \(\eta\in\Aatt\), there exists a sequence \(\eta_n\in O(\mu,\nu_n)\) such that \(\eta_n\rightharpoonup\eta\).
\end{proposition}

\begin{proof}
We fix \(\eta\in\Aatt\), and let \(\widehat\eta\) be given by Lemma~\ref{lem:Aatt-decomp}. It follows that
\[
\eta=\frac12(\id,\id)_\#\mu+\frac12\widehat\eta,
\qquad
(\mathrm{pr}_1)_\#\widehat\eta=\mu,
\qquad
(\mathrm{pr}_2)_\#\widehat\eta=S_\#\mu,
\qquad
\widehat\eta(\Sigma)=1.
\]
We now disintegrate \(\widehat\eta\) with respect to \(x_2\):
\[
\widehat\eta=\int_0^1 \widehat\eta^t\,dt,
\qquad
\widehat\eta^t(\{x_2=y_2=t\})=1
\quad\text{for a.e.\ }t\in[0,1].
\]
After the modification on a null set of \(t\), we may assume that \(\widehat\eta^t(\Sigma)=1\) for a.e.\ \(t\), since \(\widehat\eta(\Sigma)=1\) and \(\Sigma\) is Borel. Let us now choose the explicit disintegration as follows:
\[
S_\#\mu=\int_0^1\widetilde\nu^t\,dt,
\qquad
\widetilde\nu^t:=(\LL^1\!\llcorner[-1,0])\otimes\delta_t.
\]
We apply Lemma~\ref{lem:fiberwise-reduction} to \(\widehat\eta\), and after the modification on a null set of \(t\), we assume that
\[
\bar{\widehat\eta}^{\,t}:=P_\#\widehat\eta^t
\]
satisfies
\[
(\mathrm{pr}_1)_\#\bar{\widehat\eta}^{\,t}=\LL^1\!\llcorner[0,1],
\qquad
(\mathrm{pr}_2)_\#\bar{\widehat\eta}^{\,t}=\LL^1\!\llcorner[-1,0],
\qquad
\widehat\eta^t=(J_t)_\#\bar{\widehat\eta}^{\,t}
\]
for a.e.\ \(t\in[0,1]\). In particular,
\[
(\mathrm{pr}_1)_\#\widehat\eta^t
=(\LL^1\!\llcorner[0,1])\otimes\delta_t,
\qquad
(\mathrm{pr}_2)_\#\widehat\eta^t
=(\LL^1\!\llcorner[-1,0])\otimes\delta_t
\]
for a.e.\ \(t\in[0,1]\).

For each \(t\in[0,1]\), let
\[
\eta^{+,t}:=(\id,\id)_\#\bigl((\LL^1\!\llcorner[0,1])\otimes\delta_t\bigr),
\]
and define
\[
\eta_n:=\int_{A_n}\eta^{+,t}\,dt+\int_{B_n}\widehat\eta^t\,dt.
\]
Observe that this is well defined because \(t\mapsto \eta^{+,t}\) is continuous in the narrow topology, while \(t\mapsto \widehat\eta^t\) is narrowly measurable by the disintegration theorem. In other terms, for every \(\Psi\in C(K\times K)\), the maps \(t\mapsto \int \Psi\,d\eta^{+,t}\) and \(t\mapsto \int \Psi\,d\widehat\eta^t\) are measurable.

Its first marginal is
\[
(\mathrm{pr}_1)_\#\eta_n
=\int_{A_n}(\LL^1\!\llcorner[0,1]\otimes\delta_t)\,dt
+\int_{B_n}(\LL^1\!\llcorner[0,1]\otimes\delta_t)\,dt
=\int_0^1(\LL^1\!\llcorner[0,1]\otimes\delta_t)\,dt
=\mu.
\]
Similarly, we have
\[
(\mathrm{pr}_2)_\#\eta_n
=\int_{A_n}(\LL^1\!\llcorner[0,1]\otimes\delta_t)\,dt
+\int_{B_n}(\LL^1\!\llcorner[-1,0]\otimes\delta_t)\,dt
=\nu_n.
\]
Moreover, \(\eta_n(\Sigma)=1\) because both \(\eta^{+,t}\) and \(\widehat\eta^t\) are supported on \(\Sigma\). It follows that \(\eta_n\in O(\mu,\nu_n)\) by Lemma~\ref{lem:contact-set}.

Let \(\Psi\in C(K\times K)\), and set
\[
F^+(t):=\int\Psi\,d\eta^{+,t},
\qquad
F^-(t):=\int\Psi\,d\widehat\eta^t.
\]
Note that both functions belong to \(L^1(0,1)\). Since
\[
\eta=\frac12\int_0^1\eta^{+,t}\,dt+\frac12\int_0^1\widehat\eta^t\,dt,
\]
Lemma~\ref{lem:weakstar-An} yields
\[
\int\Psi\,d\eta_n
=
\int_{A_n}F^+(t)\,dt+\int_{B_n}F^-(t)\,dt
\longrightarrow
\frac12\int_0^1F^+(t)\,dt+\frac12\int_0^1F^-(t)\,dt
=
\int\Psi\,d\eta.
\]
Finally, it follows that \(\eta_n\rightharpoonup\eta\).
\end{proof}

\begin{proof}[Proof of Theorem~\ref{thm:intro-stable}]
Proposition~\ref{prop:stable-limsup} yields
\[
\limsup_{n\to\infty} O(\mu,\nu_n)\subset \Aatt,
\]
and Proposition~\ref{prop:stable-liminf} gives
\[
\Aatt\subset \liminf_{n\to\infty} O(\mu,\nu_n).
\]
Since \(K\times K\) is compact metric, \(\Pcal(K\times K)\) endowed with the narrow topology is metrisable. This means that the existence of a sequence \(\eta_n\in O(\mu,\nu_n)\) with \(\eta_n\rightharpoonup\eta\) implies \(\eta\in\liminf_{n\to\infty} O(\mu,\nu_n)\).
It holds that
\[
\lim_{n\to\infty} O(\mu,\nu_n)=\Aatt.
\]
\end{proof}

\section{Constrained \texorpdfstring{\(\Gamma\)}{Gamma}-convergence and the quadratic perturbation}\label{sec:variational}
Observe that the integrand is continuous, so the \(\Gamma\)-convergence analysis reduces to the behaviour of the constraint sets \(O(\mu,\nu_n)\) and their limit class \(\Aatt\). We first consider the constrained \texorpdfstring{\(\Gamma\)}{Gamma}-convergence on the attainable class. To this end, let \(\Phi:[0,2]\to\mathbb R\) be continuous, and let
\[
\widetilde\Phi:[0,\sqrt5]\to\mathbb R
\]
be any continuous extension of \(\Phi\). Define
\[
g_\Phi(x,y):=\widetilde\Phi(|x-y|),
\qquad
(x,y)\in K\times K.
\]
By Lemma~\ref{lem:contact-set}, we conclude that every plan in \(O(\mu,\nu_n)\cup O(\mu,\nu)\) is supported on
\[
\Sigma=\{(x,y)\in K\times K:\ x_2=y_2,\ y_1\le x_1\},
\]
and so \(|x-y|\le 2\) on the support of every admissible plan that we consider below.
We deduce that the values of the functionals \(\mathcal F_n\) and \(\mathcal F_\infty\) do not depend on the chosen extension \(\widetilde\Phi\). Let us set
\begin{equation}\label{eq:gamma-Fn}
\mathcal F_n(\eta):=
\begin{cases}
\displaystyle \int_{K\times K} g_\Phi(x,y)\,d\eta, & \eta\in O(\mu,\nu_n),\\[0.3em]
+\infty, & \text{otherwise},
\end{cases}
\end{equation}
\begin{equation}\label{eq:gamma-Finf}
\mathcal F_\infty(\eta):=
\begin{cases}
\displaystyle \int_{K\times K} g_\Phi(x,y)\,d\eta, & \eta\in\Aatt,\\[0.3em]
+\infty, & \text{otherwise}.
\end{cases}
\end{equation}

\begin{proof}[Proof of Theorem~\ref{thm:main-gamma}]
Since \(g_\Phi\in C(K\times K)\), narrow convergence implies convergence of the integrals of \(g_\Phi\). For the \(\Gamma\)-\(\liminf\) inequality, let \(\eta_n\rightharpoonup\eta\) in \(\Pcal(K\times K)\). If \(\liminf_{n\to\infty}\mathcal F_n(\eta_n)=+\infty\), there is nothing left to prove. After we extract a subsequence that realises the liminf, we may assume that \(\sup_n\mathcal F_n(\eta_n)<+\infty\). We deduce that \(\eta_n\in O(\mu,\nu_n)\) for every \(n\), and Theorem~\ref{thm:intro-stable} yields \(\eta\in\Aatt\). As a result, we obtain
\[
\mathcal F_\infty(\eta)
=
\int_{K\times K} g_\Phi\,d\eta
=
\lim_{n\to\infty}\int_{K\times K} g_\Phi\,d\eta_n
=
\liminf_{n\to\infty}\mathcal F_n(\eta_n).
\]

Let us fix \(\eta\in\Pcal(K\times K)\) for the \(\Gamma\)-\(\limsup\) inequality. If \(\eta\notin\Aatt\), then \(\mathcal F_\infty(\eta)=+\infty\), so there is nothing to show. If \(\eta\in\Aatt\), Proposition~\ref{prop:stable-liminf} provides a recovery sequence \(\eta_n\in O(\mu,\nu_n)\) such that \(\eta_n\rightharpoonup\eta\). We have
\[
\lim_{n\to\infty}\mathcal F_n(\eta_n)
=
\lim_{n\to\infty}\int_{K\times K} g_\Phi\,d\eta_n
=
\int_{K\times K} g_\Phi\,d\eta
=
\mathcal F_\infty(\eta).
\]
\end{proof}

\begin{corollary}[Homogenised minimiser for strictly convex energies]\label{cor:hom-minimizer}
In addition, we assume that \(\Phi\) is strictly convex on \([0,2]\).
It follows that \(\mathcal F_\infty\) has a unique minimiser, namely,
\[
\gamma^{\mathrm{hom}}:=\frac12(\id,\id)_\#\mu+\frac12(\id,S)_\#\mu,
\]
and
\[
\min_{\eta\in\Pcal(K\times K)}\mathcal F_\infty(\eta)
=\frac12\,\Phi(0)+\frac12\,\Phi(1).
\]
The above implies that every sequence of minimisers \((\eta_n)\) of \(\mathcal F_n\) converges narrowly to \(\gamma^{\mathrm{hom}}\).
\end{corollary}

\begin{proof}
Let \(\eta\in\Aatt\), and let \(\widehat\eta\) be as in Lemma~\ref{lem:Aatt-decomp}. We then have
\[
\eta=\frac12(\id,\id)_\#\mu+\frac12\widehat\eta,
\]
so
\[
\mathcal F_\infty(\eta)
=\frac12\Phi(0)+\frac12\int_{K\times K}\Phi(|x-y|)\,d\widehat\eta(x,y).
\]
We apply Lemma~\ref{lem:fiberwise-reduction} and Corollary~\ref{cor:fiberwise-decomposition} to \(\widehat\eta\), with \(\widetilde\nu=S_\#\mu\), and obtain a disintegration
\[
\widehat\eta=\int_0^1\widehat\eta^t\,dt
\]
such that, for a.e.\ \(t\), the one-dimensional coupling
\[
\bar{\widehat\eta}^{\,t}:=P_\#\widehat\eta^t
\]
has marginals \(\LL^1\!\llcorner[0,1]\) and \(\LL^1\!\llcorner[-1,0]\). Moreover, it is supported on \(\{y\le x\}\). By Lemma~\ref{lem:fiber-strict-convex},
\[
\int \Phi(|x-y|)\,d\bar{\widehat\eta}^{\,t}(x,y)\ge \Phi(1)
\quad\text{for a.e.\ }t,
\]
with equality if and only if \(y=x-1\) \(\bar{\widehat\eta}^{\,t}\)-a.s.
We integrate in \(t\) and obtain
\[
\mathcal F_\infty(\eta)\ge \frac12\Phi(0)+\frac12\Phi(1).
\]
Observe that equality holds if and only if
\[
\bar{\widehat\eta}^{\,t}=(\id,\id-1)_\#(\LL^1\!\llcorner[0,1])
\quad\text{for a.e.\ }t.
\]
By the reconstruction formula in Lemma~\ref{lem:fiberwise-reduction}, this is equivalent to
\[
\widehat\eta=(\id,S)_\#\mu.
\]
Hence
\[
\eta=\frac12(\id,\id)_\#\mu+\frac12(\id,S)_\#\mu
=\gamma^{\mathrm{hom}},
\]
so \(\gamma^{\mathrm{hom}}\) is the unique minimiser and
\[
\min_{\eta\in\Pcal(K\times K)}\mathcal F_\infty(\eta)
=\frac12\Phi(0)+\frac12\Phi(1).
\]

Let \((\eta_n)\) be minimisers of \(\mathcal F_n\). By compactness, every subsequence admits a narrowly convergent subsubsequence \(\eta_{n_k}\rightharpoonup\bar\eta\). The fundamental theorem of \(\Gamma\)-convergence implies that \(\bar\eta\) minimises \(\mathcal F_\infty\). By the uniqueness we proved above, \(\bar\eta=\gamma^{\mathrm{hom}}\). It follows that every cluster point of \((\eta_n)\) equals \(\gamma^{\mathrm{hom}}\), and therefore \(\eta_n\rightharpoonup\gamma^{\mathrm{hom}}\).
\end{proof}

\begin{remark}[Interpretation of the \(\Gamma\)-limit]
One could also understand the counterexample at the level of sets or at the level of functionals. Notice that at the set level, the optimal-plan sets \(O(\mu,\nu_n)\) converge to the proper subset \(\Aatt\subsetneq O(\mu,\nu)\). At the variational level, the effective \(\Gamma\)-limit is obtained by restricting the secondary problem to \(\Aatt\), rather than to the full optimal-plan set \(O(\mu,\nu)\). In this sense, the microscopic oscillations ``survive'' in the limit through the constrained admissible class.
\end{remark}

\subsection{The additive quadratic perturbation}\label{sec:noncommute}
We now consider the regularised cost
\[
c_\varepsilon(x,y)=|x-y|+\varepsilon|x-y|^2.
\]

\begin{lemma}[Uniqueness of the \(c_\varepsilon\)-optimal plan for \((\mu,\nu)\)]\label{lem:ceps-unique-limit}
For every \(\varepsilon>0\), the optimal transport problem for the cost \(c_\varepsilon\) between \((\mu,\nu)\) has a unique minimiser, which coincides with \(\gamma^{\mathrm{sel}}(\mu,\nu)\) from Proposition~\ref{prop:sec-limit}.
\end{lemma}

\begin{proof}
We fix \(\varepsilon>0\), let \(\eta\in\Pi(\mu,\nu)\), and set \(\bar\eta:=P_\#\eta\). We then obtain
\[
(\mathrm{pr}_1)_\#\bar\eta=\LL^1\!\llcorner[0,1],\qquad
(\mathrm{pr}_2)_\#\bar\eta=\tfrac12\,\LL^1\!\llcorner[-1,1].
\]
By Lemma~\ref{lem:fiber-quadratic}(ii),
\[
\int_{\mathbb R^2}(x-y)^2\,d\bar\eta(x,y)\ge \frac13,
\]
Since
\[
|x-y|^2=(x_1-y_1)^2+(x_2-y_2)^2\ge (x_1-y_1)^2,
\]
it follows that
\[
C_2(\eta)\ge \frac13.
\]
Together with \(C_1(\eta)\ge m_1(\mu,\nu)=\frac12\), this gives
\[
\int c_\varepsilon\,d\eta
=
C_1(\eta)+\varepsilon C_2(\eta)
\ge
\frac12+\frac{\varepsilon}{3}.
\]
By Lemma~\ref{lem:contact-set} and Proposition~\ref{prop:sec-limit},
\[
C_1\bigl(\gamma^{\mathrm{sel}}(\mu,\nu)\bigr)=\frac12,
\qquad
C_2\bigl(\gamma^{\mathrm{sel}}(\mu,\nu)\bigr)=\frac13,
\]
so \(\gamma^{\mathrm{sel}}(\mu,\nu)\) attains this lower bound and is \(c_\varepsilon\)-optimal.

If \(\eta\) is \(c_\varepsilon\)-optimal, then equality holds above. We have:
\[
\eta\in O(\mu,\nu)
\qquad\text{and}\qquad
C_2(\eta)=\frac13.
\]
Proposition~\ref{prop:sec-limit} implies
\[
\eta=\gamma^{\mathrm{sel}}(\mu,\nu).
\]
We conclude that the \(c_\varepsilon\)-optimal plan is unique and coincides with \(\gamma^{\mathrm{sel}}(\mu,\nu)\).
\end{proof}

\begin{proof}[Proof of Corollary~\ref{cor:quadratic}]
We analyse the two iterated limits separately.

\textit{First \(\varepsilon\downarrow0\), then \(n\to\infty\).}
Let us fix \(n\in\mathbb N\), let \(\varepsilon_k\downarrow0\), and after passing to a subsequence, we assume that \(\gamma_{n,\varepsilon_k}\rightharpoonup\bar\gamma_n\). Note that the marginals are fixed, so \(\bar\gamma_n\in\Pi(\mu,\nu_n)\). For every \(\eta\in\Pi(\mu,\nu_n)\), minimality gives
\[
C_1(\gamma_{n,\varepsilon_k})+\varepsilon_k C_2(\gamma_{n,\varepsilon_k})
\le
C_1(\eta)+\varepsilon_k C_2(\eta).
\]
Because \(K\times K\) is compact, \(C_2\) is bounded on \(\Pcal(K\times K)\). If we now let \(k\to\infty\), we obtain \(C_1(\bar\gamma_n)\le C_1(\eta)\), so \(\bar\gamma_n\in O(\mu,\nu_n)\). If now \(\eta\in O(\mu,\nu_n)\), then
\[
C_1(\eta)=m_1(\mu,\nu_n)\le C_1(\gamma_{n,\varepsilon_k}).
\]
The same inequality gives
\[
\varepsilon_k C_2(\gamma_{n,\varepsilon_k})
\le
C_1(\eta)-C_1(\gamma_{n,\varepsilon_k})+\varepsilon_k C_2(\eta)
\le
\varepsilon_k C_2(\eta),
\]
and so
\[
C_2(\gamma_{n,\varepsilon_k})\le C_2(\eta).
\]
We now pass to the limit and obtain \(C_2(\bar\gamma_n)\le C_2(\eta)\). Proposition~\ref{prop:sec-n} implies \(\bar\gamma_n=\gamma_n\). We infer that
\[
\gamma_{n,\varepsilon}\rightharpoonup\gamma_n
\qquad\text{as }\varepsilon\downarrow0
\]
for each fixed \(n\). Lemma~\ref{lem:conv-nu-gamma} yields
\[
\lim_{n\to\infty}\lim_{\varepsilon\downarrow0}\gamma_{n,\varepsilon}
=
\lim_{n\to\infty}\gamma_n
=
\gamma^{\mathrm{hom}}.
\]

\textit{First \(n\to\infty\), then \(\varepsilon\downarrow0\).}
We fix \(\varepsilon>0\), and let \(\gamma_{n_k,\varepsilon}\rightharpoonup\bar\gamma_\varepsilon\) along a subsequence. It follows that \(\bar\gamma_\varepsilon\in\Pi(\mu,\nu)\). For any \(\eta\in\Pi(\mu,\nu)\), Lemma~\ref{lem:recovery} provides \(\eta_n\in\Pi(\mu,\nu_n)\) such that \(\eta_n\rightharpoonup\eta\) and
\[
\int c_\varepsilon\,d\eta_n \to \int c_\varepsilon\,d\eta.
\]
Since \(\gamma_{n,\varepsilon}\) is \(c_\varepsilon\)-optimal for \((\mu,\nu_n)\),
\[
\int c_\varepsilon\,d\gamma_{n,\varepsilon}
\le
\int c_\varepsilon\,d\eta_n.
\]
We pass to the limit along \(n_k\) and obtain
\[
\int c_\varepsilon\,d\bar\gamma_\varepsilon
\le
\int c_\varepsilon\,d\eta
\qquad
\forall\,\eta\in\Pi(\mu,\nu).
\]
It implies that \(\bar\gamma_\varepsilon\) is \(c_\varepsilon\)-optimal for \((\mu,\nu)\). By Lemma~\ref{lem:ceps-unique-limit}, this optimal plan is unique and equals \(\gamma^{\mathrm{sel}}(\mu,\nu)\). Therefore,
\[
\gamma_{n,\varepsilon}\rightharpoonup\gamma^{\mathrm{sel}}(\mu,\nu)
\qquad\text{as }n\to\infty
\]
for each fixed \(\varepsilon>0\), and also
\[
\lim_{\varepsilon\downarrow0}\lim_{n\to\infty}\gamma_{n,\varepsilon}
=
\gamma^{\mathrm{sel}}(\mu,\nu).
\]
\end{proof}

\appendix
\section{Auxiliary lemmas}\label{sec:appendix}

\subsection{Recovery}
\begin{lemma}[Recovery under varying second marginal]\label{lem:recovery}
Let \((Z,d)\) be a compact metric space and fix \(\mu\in\Pcal(Z)\).
If \(\nu_n\rightharpoonup\nu\) in \(\Pcal(Z)\), then for every \(\eta\in\Pi(\mu,\nu)\) there exists \(\eta_n\in\Pi(\mu,\nu_n)\) such that \(\eta_n\rightharpoonup\eta\) in \(\Pcal(Z\times Z)\). Moreover, for every continuous \(c:Z\times Z\to\mathbb R\), the following holds:
\[
\int c\,d\eta_n \xrightarrow[n\to\infty]{} \int c\,d\eta.
\]
\end{lemma}

\begin{proof}
Because \(Z\) is compact, narrow convergence \(\nu_n\rightharpoonup\nu\) is equivalent to
\(W_1(\nu_n,\nu)\to0\), where \(W_1\) denotes the \(1\)-Wasserstein distance induced by \(d\). We choose \(\lambda_n\in\Pi(\nu,\nu_n)\) such that
\[
\int_{Z\times Z} d(y,z)\,d\lambda_n(y,z)=W_1(\nu,\nu_n)\xrightarrow[n\to\infty]{}0.
\]
Let us now disintegrate \(\eta(dx,dy)=\eta_y(dx)\,\nu(dy)\), and write
\[
\lambda_n(dy,dz)=k_n(y,dz)\,\nu(dy),
\qquad
\rho_n(dx,dy,dz):=\eta_y(dx)\,k_n(y,dz)\,\nu(dy),
\qquad
\eta_n:=(\mathrm{pr}_{1,3})_\#\rho_n.
\]
We use \(\mathrm{pr}_{1,3}:Z^3\to Z\times Z\) to denote the projection \(\mathrm{pr}_{1,3}(x,y,z):=(x,z)\). It follows that \(\eta_n\in\Pi(\mu,\nu_n)\): its first marginal is \(\mu\), because \(\int \eta_y\,\nu(dy)=\eta\), and its second marginal is \(\nu_n\), because the \((y,z)\)-marginal of \(\rho_n\) is \(\lambda_n\).

Let \(c\in C(Z\times Z)\). We equip \(Z\times Z\) with the sum metric
\[
d_\times((x,y),(x',y')):=d(x,x')+d(y,y').
\]
Given that \(Z\times Z\) is compact, \(c\) is uniformly continuous with respect to \(d_\times\). Let \(\omega_c\) be a modulus of continuity of \(c\) for \(d_\times\).
It follows \(d_\times((x,z),(x,y))=d(y,z)\), and so
\[
\left|\int c\,d\eta_n-\int c\,d\eta\right|
=\left|\int_{Z^3}\!\big(c(x,z)-c(x,y)\big)\,d\rho_n\right|
\le \int_{Z\times Z}\!\omega_c\!\big(d(y,z)\big)\,d\lambda_n(y,z).
\]
We fix \(\varepsilon>0\), choose \(\delta>0\) such that \(\omega_c(r)\le\varepsilon\) for \(0\le r\le\delta\), and note that
\[
\int\omega_c(d(y,z))\,d\lambda_n
\le \varepsilon+2\|c\|_\infty\,\lambda_n(\{d(y,z)>\delta\})
\le \varepsilon+\frac{2\|c\|_\infty}{\delta}\int d(y,z)\,d\lambda_n
\]
\[
=\varepsilon+\frac{2\|c\|_\infty}{\delta}\,W_1(\nu,\nu_n).
\]
If we let \(n\to\infty\), and then \(\varepsilon\downarrow0\), we obtain \(\int c\,d\eta_n\to\int c\,d\eta\). In particular, \(\eta_n\rightharpoonup\eta\) is true.
\end{proof}

\subsection{Strip averaging}
\begin{lemma}[Weak-\(*\) limit of \(\mathbf 1_{A_n}\)]\label{lem:weakstar-An}
The characteristic functions \(\mathbf 1_{A_n}\) converge to \(\frac12\) weak-\(*\) in \(L^\infty([0,1])\). In other words, for every \(f\in L^1([0,1])\),
\[
\int_{A_n} f(t)\,dt \xrightarrow[n\to\infty]{}\frac12\int_0^1 f(t)\,dt,
\qquad
\int_{B_n} f(t)\,dt \xrightarrow[n\to\infty]{}\frac12\int_0^1 f(t)\,dt.
\]
\end{lemma}

\begin{proof}
By symmetry, it suffices to prove the claim for \(A_n\). To this end, let \(f\in L^1([0,1])\) and \(\varepsilon>0\). Let us choose \(g\in C([0,1])\) such that
\(\|f-g\|_{L^1([0,1])}\le \varepsilon\). We also write \(I_{k,n}:=[k/n,(k+1)/n]\). It follows that \(A_n\cap I_{k,n}=[k/n,k/n+1/(2n)]\). Since \(g\) is uniformly continuous, if we let \(\omega_g(\delta):=\sup\{|g(s)-g(t)|:\ |s-t|\le \delta\}\), we obtain \(\omega_g(\delta)\to0\) as \(\delta\downarrow0\).
By an elementary estimate, we have
\[
\left|\int_{A_n} g(t)\,dt-\frac12\int_0^1 g(t)\,dt\right|
\le \frac14\,\omega_g(1/n)\xrightarrow[n\to\infty]{}0.
\]
Therefore,
\begin{align*}
\left|\int_{A_n} f-\frac12\int_0^1 f\right|
&\le \left|\int_{A_n} (f-g)\right|
 + \left|\int_{A_n} g-\frac12\int_0^1 g\right|
 + \frac12\left|\int_0^1 (g-f)\right|\\
&\le \|f-g\|_{L^1} + \left|\int_{A_n} g-\frac12\int_0^1 g\right| + \frac12\|g-f\|_{L^1}.
\end{align*}
We now take \(\limsup_{n\to\infty}\) to obtain \(\limsup_{n\to\infty}\big|\int_{A_n} f-\frac12\int_0^1 f\big|\le \tfrac32\varepsilon\).
Since \(\varepsilon\) is arbitrary, \(\int_{A_n} f\to \frac12\int_0^1 f\) holds.
The identity for \(B_n\) follows immediately from \(\int_{B_n} f=\int_0^1 f-\int_{A_n} f\).
\end{proof}

\bibliographystyle{plain}
\bibliography{open2_refs}

\end{document}